\documentclass[psamsfonts]{amsproc}

\usepackage[utf8]{inputenc}
\usepackage[T1]{fontenc}

\usepackage{amssymb}

\usepackage{microtype}

\newtheorem{thm}{Theorem}[section]
\newtheorem{proposition}[thm]{Proposition}
\newtheorem{lemma}[thm]{Lemma}
\newtheorem{cor}[thm]{Corollary}
\newtheorem*{problem*}{Problem}

\theoremstyle{definition}
\newtheorem{ex}[thm]{Example}
\newtheorem{remark}[thm]{Remark}

\newcommand {\zmod}[1]{\mathbb{Z}/_{\!#1}}

\title[On equivariant and invariant topological complexity]
{On equivariant and invariant topological complexity of smooth $\mathbb{Z}/\!_p$-spheres}
\author{Zbigniew Błaszczyk}
\email{blaszczyk@amu.edu.pl}
\address{Faculty of Mathematics and Computer Science\\
Adam Mickiewicz University\\
Umultowska 87\\
61-614 Poznań, Poland}
\thanks{The authors have been supported by the National Science Centre grants: \textsc{2014/12/S/ST1/00368} and \textsc{2015/19/B/ST1/01458}, respectively.}

\author{Marek Kaluba}
\email{kalmar@amu.edu.pl}
\address{Faculty of Mathematics and Computer Science\\
Adam Mickiewicz University\\
Umultowska 87\\
61-614 Poznań, Poland}
\address{
Institute of Mathematics\\
Polish Academy of Sciences\\
Śniadeckich 8\\
00-656 Warszawa, Poland
}


\begin{document}

\begin{abstract}
We investigate equivariant and invariant topological complexity of spheres endowed with smooth non-free actions of cyclic groups of prime order. We prove that semilinear $\mathbb{Z}/_{\!p}$-spheres have both invariants either $2$ or~$3$ and calculate exact values in all but two cases. On the other hand, we exhibit examples which show that these invariants can be arbitrarily large in the class of smooth $\mathbb{Z}/_{\!p}$-spheres.
\end{abstract}

\maketitle

\section{Introduction}

Consider the space $X$ of all possible configurations of a mechanical system. The motion planning problem is to describe a continuous algorithm (``motion planner'') which, given a pair \mbox{$(x,y) \in X \times X$}, outputs a continuous path in~$X$ between $x$ and $y$.
In order to measure discontinuity of the process of motion planning, Farber \cite{Farber2003a} introduced the notion of topological complexity of $X$ in the following manner. An open subset $U \subseteq X \times X$ is a \textit{domain of continuity} if a motion planner exists over~$U$; \textit{topological complexity} of $X$ is the minimal number of domains of continuity whose union covers $X \times X$. (See Section \ref{sect:TC} for a more verbose definition.)
Due to its applications in topological robotics --- its knowledge is of practical use when designing optimal motion planners --- and close relation to Lusternik--Schnirelmann category, topological complexity has attracted plenty of attention in recent years.

Thus it is perhaps natural that there appeared versions of topological complexity aimed at exploiting the presence of a group action: ``equivariant topological complexity'' ($TC_G$) and ``invariant topological complexity'' ($TC^G$). The former was defined by Colman and Grant~\cite{Colman2012}, the latter by Lubawski and Mar\-zan\-to\-wicz \cite{Lubawski2014}. The founding ideas of these two invariants are quite different. Roughly speaking, $TC_G$ records the minimal number of domains of continuity which preserve symmetries,
while $TC^G$ tries to take advantage of symmetries to ease the effort of motion planning. It seems that both have their strengths and quirks. (See Sections \ref{sect:ETC} and \ref{sect:ITC} for more details.)


The aim of this paper is to investigate equivariant and invariant topological complexity of spheres endowed with linear, and then general smooth $\zmod{p}$\nobreakdash-\hspace{0pt}actions ($p$~is any prime). We prove that if the fixed point set of a smooth $\zmod{p}$\nobreakdash-\hspace{0pt}action on~$S^n$ is homeomorphic to $S^k$, $0<k\leq n$, then, subject to a minor technical assumption if $k=n-2$, both invariants are equal to either $2$ or~$3$ (Proposition~\ref{prop:semilinear}). This happens, in particular, whenever the action is linear. In that case we calculate exact values of $TC_{\zmod{p}}$ and $TC^{\zmod{p}}$, with two exceptions for the latter invariant (Theorems \ref{thm:ETC-of-spheres} and \ref{thm:ITC-of-spheres}, respectively).

On the other hand,
given any $n\geq 5$, we construct a smooth $\zmod{p}$\nobreakdash-\hspace{0pt}sphere~$S^n$ with the fixed point set an essential homology $(n-2)$\nobreakdash-\hspace{0pt}sphere, thus showing that $TC_{\zmod{p}}$ and $TC^{\zmod{p}}$ can be arbitrarily large in the class of smooth $\zmod{p}$\nobreakdash-\hspace{0pt}spheres with non-empty fixed point sets (Proposition \ref{prop:ITC-arbitrarily-large}).

The paper is organized as follows. Section \ref{sec:Preliminaries} is devoted to preliminaries. We review definitions of topological complexity and its equivariant counterparts, and compare some of the relevant properties of $TC_G$ and $TC^G$. In particular, we provide a lower bound for $TC^G$ in terms of Lusternik--Schnirelmann $G$\nobreakdash-\hspace{0pt}category, a result that was previously known for $TC_G$. In Section \ref{sec:Linear-spheres} we take a closer look at behaviour of $TC_G$ and $TC^G$ on spheres with linear actions of any finite group $G$, and then specialize to the case $G= \zmod{p}$. In Section \ref{sec:Smooth-actions} we explore both invariants in the realm of general smooth $\zmod{p}$\nobreakdash-\hspace{0pt}actions on spheres. We also hint at possible further direction of research. In the appendix we briefly discuss the existence of nowhere-vanishing equivariant vector fields --- this is indispensable for our proof of Theorem \ref{thm:ETC-of-spheres}.\medskip

\noindent\textbf{Notation.} We use the non-reduced version of Lusternik--Schni\-rel\-mann category (``LS category''), topological complexity, and their equivariant counterparts, so that, for example, $\textnormal{cat}(S^n)=2$.

\section{Preliminaries}\label{sec:Preliminaries}

Throughout this section $G$ stands for a compact Hausdorff topological group.

\subsection{Topological complexity}\label{sect:TC}

Let $PX$ be the space of all (continuous) paths in a topological space $X$, with the compact-open topology. Write $\pi \colon PX \to X \times X$ for the path fibration
\[ \pi(\gamma) = \big(\gamma(0), \gamma(1)\big) \textnormal{ for $\gamma \in PX$,} \]
associating to any path $\gamma \in PX$ its initial and terminal points. A \textit{motion planner} on an open subset $U \subseteq X \times X$ is defined to be a local section of $\pi$, i.e., a map $s \colon U \to PX$ such that $\pi \circ s$ is the inclusion $i_U \colon U \hookrightarrow X\times X$.

\textit{Topological complexity} of $X$, denoted $TC(X)$, is the least integer $k\geq 1$ such that there exists an open cover of $X \times X$ by $k$ sets which admit motion planners. Clearly, a necessary condition for $TC(X)$ to be finite is path-connectedness of~$X$.

Topological complexity is well known to be invariant under homotopy equivalences. It is straightforward to see that a ``global'' motion planner on $X$ exists if and only if $X$ is contractible or, in other words, $TC(X)=1$ if and only if $X$ is contractible. Furthermore, Grant, Lupton and Oprea proved:

\begin{thm}[{\cite[Corollary 3.5]{Grant2013}}]\label{TC_2}
Let $X$ be a path-connected CW-complex with finitely many cells in each dimension. Then $TC(X)=2$ if and only if $X$ is homotopy equivalent to an odd-dimensional sphere.
\end{thm}

\begin{ex}[{\cite[Theorem 8]{Farber2003a}}]\label{TC_spheres}
Recall that topological complexity of the sphere~$S^n$, $n\geq 1$, is given by:
\[ TC(S^n) = \begin{cases}
2, & \textnormal{$n$ odd,}\\
3, & \textnormal{$n$ even.}
\end{cases} \]
\end{ex}

We conclude this section by recalling that topological complexity is closely related to LS category. In fact, for any path-connected and paracompact space~$X$, we have:
\[ \textnormal{cat}(X) \leq TC(X) \leq 2\,\textnormal{cat}(X)-1.\]





\subsection{Equivariant topological complexity $TC_G$}\label{sect:ETC}

Given a $G$\nobreakdash-\hspace{0pt}space X, we will view $PX$ and $X\times X$ as $G$\nobreakdash-\hspace{0pt}spaces via the formulas:
\[ (g\gamma)(-)=g\big(\gamma(-)\big) \textnormal{ for $g \in G$ and $\gamma \in PX$,} \]
and
\[ g(x,y)=(gx,gy) \textnormal{ for $g\in G$, $x \in X$ and $y\in Y$,} \]
respectively. The path fibration $\pi \colon PX \to X\times X$ is then a $G$\nobreakdash-\hspace{0pt}fibration.

\textit{Equivariant topological complexity} of $X$, written $TC_G(X)$, is the least integer $k\geq 1$ such that there exists an open $G$\nobreakdash-\hspace{0pt}invariant cover of $X \times X$ by $k$ sets which admit $G$\nobreakdash-\hspace{0pt}equivariant motion planners (i.e. motion planners which are $G$\nobreakdash-\hspace{0pt}equivariant maps).

Given a closed subgroup $H \subseteq G$, let $X^H=\{x \in X \;|\; \textnormal{$hx = x$ for any $h \in H$}\}$ be the \textit{$H$\nobreakdash-\hspace{0pt}fixed point set} of $X$. We record that:

\begin{proposition}[{\cite[Corollary 5.4]{Colman2012}}]\label{CG_TCG_lower}
Let $X$ be a $G$\nobreakdash-\hspace{0pt}space. For any closed subgroup $H \subseteq G$, we have
\[ TC(X^H)\leq TC_G(X). \]
In particular, $TC(X)\leq TC_G(X)$.
\end{proposition}

\noindent This shows that in order for $TC_G(X)$ to be finite, we at least need to assume that $X$ is a \textit{$G$\nobreakdash-\hspace{0pt}connected} space, i.e., its $H$\nobreakdash-\hspace{0pt}fixed point sets are path-connected for all closed subgroups $H\subseteq G$.

Equivariant topological complexity is invariant under equivariant homotopy equivalences. It also relates to Lusternik--Schnirelmann $G$\nobreakdash-\hspace{0pt}category of $X$, $\textnormal{cat}_G(X)$, in the same way that topological complexity relates to LS category, at least when $X$ is a $G$\nobreakdash-\hspace{0pt}connected space with a non-empty fixed point set $X^G$. (Recall that $\textnormal{cat}_G (X)$ is defined as follows. A $G$\nobreakdash-\hspace{0pt}invariant subset $U \subseteq X$ is said to be \textit{$G$\nobreakdash-\hspace{0pt}categorical} if the inclusion $U \hookrightarrow X$ is $G$\nobreakdash-\hspace{0pt}homotopic to a map with values in a single orbit; then $\textnormal{cat}_G(X)$ is the least integer $k\geq 1$ such that there exists an open cover of $X$ by $k$ sets that are $G$\nobreakdash-\hspace{0pt}categorical.  In particular, $\textnormal{cat}_G(X) =1$ if and only if $X$ is $G$\nobreakdash-\hspace{0pt}contractible.) More precisely, we have:

\begin{proposition}[{\cite[Corollary 5.8]{Colman2012}}]\label{CG_TCG_Gcat}
Let $X$ be a completely normal $G$\nobreakdash-\hspace{0pt}space. If $X$ is $G$\nobreakdash-\hspace{0pt}connected and $X^G\neq \emptyset$, then
\[ \textnormal{cat}_G(X) \leq TC_G(X) \leq 2\,\textnormal{cat}_G(X)-1. \]
\end{proposition}

\subsection{Invariant topological complexity $TC^G$}\label{sect:ITC}

For a $G$\nobreakdash-\hspace{0pt}space $X$ define
\[ PX \times_{X/G} PX = \big\{(\gamma,\delta) \in PX \times PX \,\big|\, G\gamma(1)=G\delta(0) \big\}. \]
Clearly, both $PX \times_{X/G} PX$ and $X\times X$ are $(G\times G)$\nobreakdash-\hspace{0pt}spaces with respect to the actions:
\[ (g_1,g_2)(\gamma,\delta)=(g_1\gamma, g_2\delta) \textnormal{ for $g_1$, $g_2 \in G$ and $(\gamma,\delta) \in PX\times_{X/G} PX$,} \]
where the $G$\nobreakdash-\hspace{0pt}action on each copy of $PX$ is the one given in the previous section, and
\[ (g_1, g_2)(x,y) = (g_1x,g_2y) \textnormal{ for $g_1$, $g_2 \in G$ and $x$, $y \in X$.} \]
Then the map $\pi^G \colon PX \times_{X/G} PX \to X\times X$ given by
\[ \pi^G(\gamma,\delta) = \big(\gamma(0),\delta(1)\big) \textnormal{ for $(\gamma, \delta) \in PX\times_{X/G} PX$}\]
turns out to be a $(G\times G)$\nobreakdash-\hspace{0pt}fibration.

\textit{Invariant topological complexity} of $X$, written $TC^G(X)$, is the least integer $k\geq 1$ such that there exists an open $(G\times G)$\nobreakdash-\hspace{0pt}invariant cover $U_1$, \ldots, $U_k$ of $X \times X$ and $(G\times G)$\nobreakdash-\hspace{0pt}equivariant maps $s_i \colon U_i \to PX\times_{X/G} PX$ with $\pi^G \circ s_i = i_{U_i}$, $1 \leq i \leq k$. Note that the $s_i$'s are not motion planners in the sense of (non-equivariant) topological complexity.

Another, equivalent, approach to $TC^G$ is as follows. Let $A \subseteq X$ be a closed $G$\nobreakdash-\hspace{0pt}invariant subset. Recall that an open $G$\nobreakdash-\hspace{0pt}invariant subspace $U \subseteq X$ is said to be \textit{$G$\nobreakdash-\hspace{0pt}compressible} into $A$ if the inclusion $U \hookrightarrow X$ is $G$\nobreakdash-\hspace{0pt}homotopic to a $G$\nobreakdash-\hspace{0pt}equivariant map $c \colon U \to X$ such that $c(U)\subseteq A$. Then $TC^G$ can be expressed as the least integer $k\geq 1$ such that there exists an open cover of $X\times X$ by $k$ sets that are $(G\times G)$\nobreakdash-\hspace{0pt}compressible into $\daleth(X) = (G\times G)\Delta(X)$, where $\Delta(X)$ stands for the diagonal of $X\times X$. (Descriptions in these terms also exist for $TC$ and $TC_G$, but we will not make use of them.)

Just as equivariant topological complexity, invariant topological complexity is invariant under equivariant homotopy equivalences.  Furthermore:

\begin{proposition}[{\cite[Remark 3.9, Corollary 3.26]{Lubawski2014}}]\label{orbit_bound}
Let $X$ be a $G$\nobreakdash-\hspace{0pt}space with $X^G\neq\emptyset$. Then
\[ \max\!\big\{TC(X^G), TC(X/G)\big\} \leq TC^G(X). \]
\end{proposition}

As far as relationship between invariant topological complexity and Lusternik--Schnirelmann $G$\nobreakdash-\hspace{0pt}category is concerned, Lubawski and Marzantowicz proved:

\begin{proposition}[{\cite[Remark 3.24]{Lubawski2014}}]\label{LM_catG_and_TCG}
If $X$ is $G$\nobreakdash-\hspace{0pt}connected and $X^G\neq \emptyset$, then
\[ TC^G(X)\leq 2\, \textnormal{cat}_G(X)-1. \]
\end{proposition}

\noindent It turns out, however, that $TC^G(X)$ is also bounded from below by $\textnormal{cat}_G(X)$ under the above assumptions. In fact, we have:

\begin{proposition}\label{catG_and_TCG}
If $X$ is a $G$\nobreakdash-\hspace{0pt}space and $\star \in X^G$, then $_{\{\star\}}\textnormal{cat}_G(X) \leq TC^G(X)$. In particular, $\textnormal{cat}_G(X) \leq TC^G(X)$.
\end{proposition}

\noindent Here $_{\{\star\}}\textnormal{cat}_G(X)$ stands for the least integer $k\geq 1$ such that there exists an open cover of $X$ by $k$ sets that are $G$\nobreakdash-\hspace{0pt}compressible into $\{\star\}$.

\begin{proof}
Let $U_1$, \ldots, $U_k$ form a $(G \times G)$\nobreakdash-\hspace{0pt}invariant open cover of $X \times X$. Write $s_i \colon U_i \to PX \times_{X/G} PX$ for a $(G\times G)$\nobreakdash-\hspace{0pt}motion planner on $U_i$, $1 \leq i \leq k$. Choose a fixed point $\star \in X^G$ and define
\[ V_i = \big\{x \in X \;|\; (x, \star) \in U_i \big\} , \textnormal{ $1 \leq i \leq k$.}\]
It is easy to see that $V_i$'s form a $G$\nobreakdash-\hspace{0pt}invariant open cover of $X$. We claim that those of $V_i$'s which are non-empty are also $G$\nobreakdash-\hspace{0pt}compressible into $\{\star\}$. Note that $s_i(x, \star)$ is an actual path in $X$ for any $x \in V_i$. Indeed, if $x \in V_i$ and  $s_i(x,\star)=(\gamma, \delta)$, then
\[ (\gamma, g\delta) = (1,g)(\gamma,\delta) = (1,g)s_i(x, \star) = s_i(x,g\star) = s_i(x,\star) = (\gamma,\delta),\]
so that $g\delta = \delta$ for any $g \in G$, which means that $\delta$ is a path in $X^G$. In particular, since $G\gamma(1)=G\delta(0)$, we have $\gamma(1)=\delta(0)$.

It is now clear that $H_i \colon V_i \times [0,1] \to X$ given by
\[ H_i(x,t) = s_i(x,\star)(t) \textnormal{ for $x \in V_i$ and $t \in [0,1]$} \]
$G$\nobreakdash-\hspace{0pt}compresses $V_i$ into $\{\star\}$ for each $1 \leq i \leq k$.
\end{proof}

\begin{cor}\label{Gcontractibility}
Let $X$ be a $G$\nobreakdash-\hspace{0pt}space with $X^G \neq \emptyset$. Then $X$ is $G$\nobreakdash-\hspace{0pt}contractible if and only if $TC^G(X)=1$.
\end{cor}

\begin{proof}
($\Rightarrow$) Suppose that $X$ is $G$\nobreakdash-\hspace{0pt}contractible, so that there exists a $G$\nobreakdash-\hspace{0pt}homotopy $H \colon X \times [0,1] \to X$ such that $H(-,0)$ is the identity map on $X$ and $H(-,1) = c$ is a map with values in a single orbit, say $Gx_0$. Define $\tilde{H}\colon (X \times X) \times [0,1] \to X \times X$ by setting
\[ \tilde{H}\big((x,y),t\big) = \big(H(x,t), H(y,t)\big) \textnormal{ for $x$, $y \in X$ and $t\in [0,1]$.}\]
Clearly, $\tilde{H}$ is a $(G\times G)$\nobreakdash-\hspace{0pt}equivariant map, $\tilde{H}(-,0)$ is the identity on $X \times X$, and
\[ \tilde{H}\big((x,y),1\big) = \big(c(x), c(y)\big) \in Gx_0 \times Gx_0 = (G\times G)(x_0,x_0) \subseteq \daleth(X). \]
This shows that $X \times X$ is $(G\times G)$\nobreakdash-\hspace{0pt}compressible into $\daleth(X)$, hence $TC^G(X)=1$.

($\Leftarrow$) This is a straightforward consequence of Proposition \ref{catG_and_TCG}: if $TC^G(X)=1$, then $\textnormal{cat}_G(X) = 1$, which exactly means that $X$ is $G$\nobreakdash-\hspace{0pt}contractible.
\end{proof}

\begin{remark}
Note that we do not need to know that the fixed point set is non-empty in order for the ``($\Rightarrow$)'' implication to work. In other words, $G$\nobreakdash-\hspace{0pt}contractibility of~$X$ \textit{always} implies $TC^G(X)=1$, even for a disconnected $X$. Equivariant topological complexity does not have an analogous property: it can be arbitrarily high for $G$\nobreakdash-\hspace{0pt}contractible spaces, see \cite[Theorem 5.11]{Colman2012}.
\end{remark}

We conclude this section with a simple example which simultaneously shows that:
\begin{enumerate}
\item[(1)] A space $X$ needs not be $G$\nobreakdash-\hspace{0pt}connected in order for $TC^G(X)$ to be finite.
\item[(2)] Invariant topological complexity does not have a property analogous to the one stated in Proposition \ref{CG_TCG_lower}.
\end{enumerate}

\begin{ex}\label{ex:disconnected_X}
Let $G=\zmod{2}\oplus\zmod{2}=\langle a, b\rangle$. Consider $X = S^1 - \{N, S\}$ equipped with the $G$\nobreakdash-\hspace{0pt}action given by:
\[
\begin{cases}
a(x,y) = (-x,y)\\
b(x,y) = (x,-y)
\end{cases}
\textnormal{ for $(x, y) \in X$.}
\]
Clearly, $X$ is $G$\nobreakdash-\hspace{0pt}contractible, hence $TC^G(X)=1$. On the other hand, since $X^{\langle b\rangle}$ is disconnected, $TC\big(X^{\langle b\rangle}\big)=\infty$.
\end{ex}

\section{Linear spheres}\label{sec:Linear-spheres}

Recall that the fixed point set of a linear action on $S^n$ is a standard unknotted sphere,  i.e., the intersection of a linear subspace of $\mathbb{R}^{n+1}$ with $S^n$. Consequently, if the fixed point set is non-empty, there are at least two fixed points.

\subsection{Arbitrary finite groups}

Throughout this section $G$ is a finite group.

\begin{proposition}\label{linear_Gcat}
Suppose that a finite group $G$ acts linearly on a sphere $S^n$ with $(S^n)^G \neq \emptyset$. Then $\textnormal{cat}_G(S^n)=2$.
\end{proposition}

\begin{proof}
Assume without loss of generality that the last coordinate of $S^n$ is left fixed by the action. In particular, $(0,\ldots,0,\pm 1)$ are fixed points.

As $G$\nobreakdash-\hspace{0pt}contractibility of a $G$\nobreakdash-\hspace{0pt}space with a non-empty fixed point set implies its contractibility, $S^n$ is not $G$\nobreakdash-\hspace{0pt}contractible, hence $\textnormal{cat}_G(S^n) > 1$. Write $S^n_+$ for a ``collared'' northern hemisphere, i.e., the set of points $(x_0, \ldots, x_n) \in S^n$ with $x_n > -\frac{1}{2}$,
and define $S^n_-$ accordingly. Observe that both these sets are $G$\nobreakdash-\hspace{0pt}categorical. Indeed, since the action is linear and leaves the last coordinate fixed, the usual contraction of $S^n_{\pm}$ onto the north (south) pole is a $G$\nobreakdash-\hspace{0pt}equivariant map.
\end{proof}


In view of Proposition \ref{CG_TCG_Gcat} (for $TC_G$) and Propositions \ref{LM_catG_and_TCG}, \ref{catG_and_TCG} (for $TC^G$), Proposition \ref{linear_Gcat} implies:

\begin{thm}\label{linear_23}
If $S^n$ is a linear $G$\nobreakdash-\hspace{0pt}connected sphere with $(S^n)^G\neq\emptyset$, then
\[2 \leq TC_G(S^n),\, TC^G(S^n) \leq 3.\]
\end{thm}

Further properties of (invariant) topological complexity yield:

\begin{cor}\label{TC_arbitrary}
Let $S^n$ be a linear $G$\nobreakdash-\hspace{0pt}connected sphere with $(S^n)^G=S^k$, where $0< k < n$. Suppose that either:
\begin{enumerate}
\item $k$ is even, or
\item the orbit space $S^n/G$ is neither contractible nor homotopy equivalent to an odd-dimensional sphere.
\end{enumerate}
Then $TC^G(S^n)=3$.
\end{cor}

\begin{proof}
As Theorem \ref{TC_2} tells, the only spaces with topological complexity $2$ are odd-dimensional homotopy spheres. Therefore if either (1) or (2) holds, $TC^G(S^n)$ is at least $3$ by Proposition \ref{orbit_bound}, and then $3$ by Theorem \ref{linear_23}.
\end{proof}

\subsection{Cyclic groups}

We now specialize to the case $G = \zmod{p}$, where $p$ is a prime number. Note that in this case $G$\nobreakdash-\hspace{0pt}connectedness amounts to path-con\-nec\-ted\-ness of the fixed point set.

\begin{proposition}\label{TCG_2}
Let $S^n$ be a linear $\zmod{p}$\nobreakdash-\hspace{0pt}sphere such that $(S^n)^{\zmod{p}}\neq \emptyset$. Then $TC_{\zmod{p}}(S^n)=2$ if and only if both $S^n$ and $(S^n)^{\zmod{p}}$ are odd-dimensional spheres.
\end{proposition}

\begin{proof}
($\Rightarrow$) Since $(S^n)^{\zmod{p}}$ cannot be contractible,
\[ TC(S^n)=TC\big((S^n)^{\zmod{p}}\big)=2 \textnormal{ by Proposition \ref{CG_TCG_lower}}, \]
and both $S^n$ and $(S^n)^{\zmod{p}}$ are odd-dimensional spheres by Theorem \ref{TC_2}.

($\Leftarrow$) Consider Farber's two-fold cover $U_1$, $U_2$ of $S^n \times S^n$, i.e.,
\begin{align*}
U_1 &= \big\{(x, y) \in S^n \times S^n \;\big|\; x\neq -y\big\},\\
U_2 &= \big\{(x, y) \in S^n \times S^n \;\big|\; x\neq y\big\}.
\end{align*}
Clearly, both these sets are $\zmod{p}$\nobreakdash-\hspace{0pt}invariant. We claim that taking the usual motion planners $s_i \colon U_i \to PS^n$, $i=1$, $2$, will suffice, with the obvious minor alteration of using a nowhere-vanishing equivariant vector field for $U_2$. (That such a vector field exists follows from Corollary \ref{vfields}.) Let us briefly explain.

Recall that $s_1(x,y)$ is defined to be the unique shortest arc between $x$ and~$y$ passed with constant velocity. In order to conclude that $s_1$ is equivariant, it is enough to show that $gs_1(x,y)$ is the shortest path between $gx$ and $gy$ for any $g \in \zmod{p}$. But the action is linear, hence any $g \in \zmod{p}$ is an orthogonal transformation and the length $\ell$ of $gs_1(x,y)$ is the same as that of $s_1(x,y)$. Therefore:
\[ \ell\big(gs_1(x,y)\big)= \ell\big(s_1(x,y)\big) \leq \ell\big(g^{-1}s_1(gx,gy)\big) =\ell\big(s_1(gx,gy)\big), \]
which is what we wanted.

The motion planner $s_2$ is constructed in two steps. Given $(x,y) \in U_2$, first pass through the unique shortest arc between $x$ and $-y$ with constant velocity (it follows from the previous paragraph that this step is equivariant), and then move along the spherical arc
\[ \cos(\pi t) \cdot y + \sin(\pi t)\cdot v(y), \textnormal{ $t \in [0,1]$} \]
defined by a nowhere-vanishing equivariant vector field $v$.
\end{proof}

Now, Propositions \ref{CG_TCG_lower} and \ref{TCG_2} yield:

\begin{thm}\label{thm:ETC-of-spheres}
Let $S^n$ be a linear $\zmod{p}$\nobreakdash-\hspace{0pt}sphere with $(S^n)^{\zmod{p}} = S^k$, $0< k< n$.
\begin{enumerate}
\item If both $n$ and $k$ are odd, then $TC_{\zmod{p}}(S^n)=2$.
\item If either $n$ or $k$ are even, then $TC_{\zmod{p}}(S^n)=3$.
\end{enumerate}
\end{thm}


We now turn attention to invariant topological complexity.

\begin{thm}\label{thm:ITC-of-spheres}
Let $S^n$ be a linear $\zmod{p}$\nobreakdash-\hspace{0pt}sphere with $(S^n)^{\zmod{p}} = S^k$, $0< k< n$. Suppose that either:
\begin{enumerate}
\item $k<n-2$,
\item $k=n-2$ and $n$ is even, or
\item $k=n-1$ and $n$ is odd.
\end{enumerate}
Then $TC^{\zmod{p}}(S^n)=3$.
\end{thm}

\begin{remark}
The case $k=n-1$ for the involution given by reflection of the last coordinate was treated by Lubawski and Marzantowicz in \cite[Example~4.2]{Lubawski2014}.  Unfortunately, the argument therein is invalid. Let us explain.

Let $g \in \zmod{2}$ be the generator. The authors define a $(\zmod{2}\times \zmod{2})$\nobreakdash-\hspace{0pt}motion planner $s_1$ on the set
\[ U_1 = \big\{ (x, y)\in S^n\times S^n \;\big|\; x\neq -y \textnormal{ if both $x$, $y$ are fixed points}\big\} \]
as follows. Write $s'(x,y)$ for the unique shortest arc in the northern hemisphere that connects either $x$ or $gx$ with either $y$ or $gy$; suppose that the arc is between the points $\alpha x$ and $\beta y$, where \mbox{$\alpha$, $\beta \in \zmod{2}$.} Note, however, that if either $x$ or $y$ is a fixed point, say $y$ for clarity, then the arc from $\alpha x$ to $y$ is the same as the arc from $\alpha x$ to $gy$. In other words, even though $s'(x,y)$ is well-\hspace{0pt}defined, we have two choices for $\beta$: it does not really matter which element $\beta$ is. Therefore setting
\[ s_1(x,y) = \left(\alpha s'(x,y)_{\big|[0,\frac{1}{2}]}, \beta s'(x,y)_{\big|[\frac{1}{2}, 1]} \right)\]
does not produce a well-defined map.

\end{remark}

\begin{proof}[Proof of Theorem \ref{thm:ITC-of-spheres}]
A linear action on $S^n$ with a $k$-dimensional fixed point set can be seen as $\Sigma^{k+1}\big(S(V)\big)$, where $V$ denotes a free $(n-k)$-dimensional $\zmod{p}$-representation, $S(V)$ the unit sphere in~$V$, and $\Sigma^{k+1}$ the $(k+1)$-fold suspension. Thus the orbit space $S^n/G$ is homeomorphic to the suspension of the orbit space~$S(V)/G$. If $0<k<n-2$, then $S(V)/G$ is homeomorphic to a real projective space ($p=2$) or a lens space ($p>2$), hence in both cases $S^n/G$ has $n-k$ non-vanishing homology groups and the conclusion follows from Corollary \ref{TC_arbitrary}. For the other two cases Corollary \ref{TC_arbitrary} applies straightforwardly.

Alternatively, one can use methods of \cite[Chapter VII]{Bredon} to compute the homology of $S^n/G$. This approach applies to a broader class of actions (simplicial/cellular), which shows that in the discussed cases $TC^{\zmod{p}}(S^n)\geq 3$ holds in greater generality.
\end{proof}

The two cases omitted in Theorem \ref{thm:ITC-of-spheres} (actions with codimension-two fixed point sets on odd-dimensional spheres, and with codimension-one fixed point sets on even-dimensional spheres) are at present unsettled.


\section{Smooth $\zmod{p}$\nobreakdash-\hspace{0pt}spheres}\label{sec:Smooth-actions}



\textit{Results of this section are stated in terms of $TC^{\zmod{p}}$, but everything carries over verbatim to the case of $TC_{\zmod{p}}$.}\medskip

\noindent Recall that a smooth $\zmod{p}$\nobreakdash-\hspace{0pt}sphere $S^n$ is said to be \textit{semilinear} whenever:
\begin{enumerate}
\item[(1)] the fixed point set $(S^n)^{\zmod{p}}$ is homeomorphic to a standard sphere,
\item[(2)]  if $\textnormal{codim}\,(S^n)^{\zmod{p}}=2$, then $S^n - (S^n)^{\zmod{p}}$ is homotopy equivalent to $S^1$.(\footnote{This is a purely technical assumption and can be omitted more often than not.})
\end{enumerate}
Since semilinear spheres are known to be equivariantly homotopy equivalent to linear $\zmod{p}$\nobreakdash-\hspace{0pt}spheres (see \cite[Proposition 5.1]{Schultz1984}),  we have:

\begin{proposition}\label{prop:semilinear}
The results in Theorems \ref{thm:ETC-of-spheres} and \ref{thm:ITC-of-spheres} hold for semilinear $\zmod{p}$-spheres.
\end{proposition}


\begin{remark}
The equivariant homotopy equivalence in question can often be improved to equivariant \textit{topological} equivalence. For example, if
$(S^n)^{\zmod{p}}=S^k$ is a standard unknotted sphere of codimension greater than $[n/2]$, then the action on $S^n$ is (at least) topologically conjugate to a linear action.

Indeed, let $x\in S^k$ and write
$T_x S^n \cong \mathbb{R}^k\oplus D(\rho)$, where \mbox{$\rho\colon \zmod{p} \to SO(n-k)$} is a free representation. Since $\dim_{\mathbb{R}}\rho > k$, the normal bundle $\nu(S^k,S^n)$ is in stable range, thus is already trivial. Therefore there exists an invariant tubular neighbourhood of the fixed point set $X = S^k \times D(\rho)$. Decompose $S^n$ as $X\cup_{\partial} Y$. Note that the $n$\nobreakdash-\hspace{0pt}sphere in $(k+1)\mathbf{1}_{\zmod{p}} \oplus \rho$ admits a similar decomposition:
\[S\big((k+1)\mathbf{1}_{\zmod{p}} \oplus \rho\big) = S^k \times D(\rho) \cup_{S^k \times S(\rho)} D^{k+1}\times S(\rho).\]
By Alexander duality, $Y$ has the same homology as $D^{k+1}\times S(\rho)$. The inclusion map $S^k \times S(\rho)\hookrightarrow Y$ extends to a $\zmod{p}$\nobreakdash-\hspace{0pt}map $D^{k+1}\times S(\rho)\rightarrow Y$ (to see this, extend over a single cell from a $\zmod{p}$\nobreakdash-\hspace{0pt}CW-decomposition, then propagate by a free $\zmod{p}$\nobreakdash-\hspace{0pt}action), which in turn can be approximated by a $\zmod{p}$\nobreakdash-\hspace{0pt}embedding $f$ by a general position argument. Comparing the Mayer--Vietoris sequences of both decompositions shows that $f$ is a homotopy equivalence.

Now find a sequence of free $\zmod{p}$\nobreakdash-\hspace{0pt}surgeries $\operatorname{rel} \partial$ between $D^{k+1}\times S(\rho)$ and $Y$. This yields an equivariant $h$\nobreakdash-\hspace{0pt}cobordism $W$ of free manifolds; the diffeomorphism between them can be chosen to be $\zmod{p}$\nobreakdash-\hspace{0pt}equivariant if and only if its Whitehead torsion $w_W\in \operatorname{Wh}(\zmod{p})$ vanishes.

This fits into a broader picture. Consult \cite{Assadi1985} for more details on classification of actions with fixed point sets of high codimension.
\end{remark}

On the other hand, a general smooth $\zmod{p}$\nobreakdash-\hspace{0pt}sphere will typically have $TC^{\zmod{p}}$ at least $4$. Consider the following construction (cf. \cite{Hsiang1964}).

\begin{ex}\label{ex:highTC}
Let $n \geq 3$. Consider a smooth homology $n$\nobreakdash-\hspace{0pt}sphere $\Sigma$ that bounds a smooth contractible $(n+1)$\nobreakdash-\hspace{0pt}manifold, say $M$, and is not a standard sphere. (Such manifolds are well known to exist. Three-dimensional examples can be found among Brieskorn varieties, see for example \cite{Casson1981}; for higher dimensions, consult \cite[Theorem~3]{Kervaire1969}.) Take any non-trivial representation $\chi \colon \zmod{p} \to U(1)$, where $p$ is an odd prime, and endow the product \mbox{$N = M \times D(\chi)$} with the diagonal action by setting
\[ g(x,y) = (x, gy) \textnormal{ for any $g \in \zmod{p}$, $x \in M$ and $y \in D(\chi)$}. \]
Since $\dim N \geq 6$ and $N$ is contractible, the $h$\nobreakdash-\hspace{0pt}cobordism theorem asserts that $N$ is diffeomorphic to $D^{n+3}$.
Therefore restricting the action to the boundary yields a smooth $\zmod{p}$\nobreakdash-\hspace{0pt}sphere $S^{n+2}$ with $\Sigma$ as the fixed point set. The fundamental group of $\Sigma$ is clearly neither trivial nor free, therefore $\textnormal{cat}(\Sigma)\geq 4$ by \mbox{\cite[Corollary 1.2]{Dranishnikov2008}}, and, consequently,
\[ 4\leq\textnormal{cat}(\Sigma) \leq TC(\Sigma) = TC\big((S^n)^{\zmod{p}}\big) \leq TC^{\zmod{p}}(S^n). \]
\end{ex}

\begin{remark}
The fixed point set of the above action is of codimension two. This is not crucial: altering the construction by taking direct sums of $\chi$ allows for a fixed point set of any even codimension. On the other hand, it is easy to see that odd-codimension fixed point sets cannot arise at all: the normal bundle is almost complex (hence even-dimensional) when $p$ is odd.
\end{remark}

In fact, $TC^{\zmod{p}}$ can be arbitrarily high in the class of smooth $\zmod{p}$\nobreakdash-\hspace{0pt}spheres with non-empty fixed point sets. More precisely, we have:

\begin{proposition}\label{prop:ITC-arbitrarily-large}
Let $p > 2$ be a prime. For any $n\geq 5$ there exists a smooth $\zmod{p}$\nobreakdash-\hspace{0pt}sphere $S^n$ with $(S^n)^{\zmod{p}}\neq \emptyset$ such that $TC^{\zmod{p}}(S^n)\geq n-1$.
\end{proposition}

The crucial element in our approach to the proof of Proposition \ref{prop:ITC-arbitrarily-large} is the existence of essential homology spheres. While this may be a well known fact, we were unable to locate an explicit reference, hence we provide a proof for the convenience of the reader. (Recall that a closed orientable manifold $X$ is called \textit{essential} provided that there exists a map $f \colon X \to B\pi_1(X)$ such that $f_*[X] \neq 0$ in integral homology, where $B\pi_1(X)$ denotes the classifying space of $\pi_1(X)$ and $[X]$~the fundamental class of $X$.)

\begin{lemma}\label{lem:n-essential-homology-spheres}
For any $n\ge 3$ there exists an essential homology $n$-sphere.
\end{lemma}

\begin{proof}
In dimensions $3$ and $4$ there are known examples of aspherical homology spheres (see \cite{Ratcliffe}), and aspherical manifolds are clearly essential.

Let $n\geq 5$. A homology $n$-sphere $\Sigma$ with fundamental group isomorphic to a (superperfect) group $G$ defines an element in $\pi_n(BG)$ via Quillen's plus construction. Indeed, the map $\kappa\colon\Sigma \to BG$ classifying the fundamental group gives rise to $\kappa^+\colon S^n \to BG^+$, which depends only on the identification $\kappa_*\colon \pi_1(\Sigma) \to G$. According to \cite[Section 6]{Hausmann1978}, this defines a bijective correspondence between the set of $n$-homology spheres with $\pi_1 \cong G$ and $\pi_n(BG^+)$. It follows that a cycle in $H_n(BG)$ can be represented by a smooth homology sphere with fundamental group $G$ if and only if its image in $H_n(BG^+)$ is in the image of the Hurewicz homomorphism $\pi_n(BG^+) \to H_n(BG^+)$.

Now choose a finitely presented group $G_n$ such that $H_*(BG_n)\cong H_*(S^n)$. (Such a group exists by \cite[Corollary B1]{Baumslag1983}.) Observe that each~$G_n$ is a superperfect group. Then the space $BG_n^+$ is $(n-1)$-connected, hence the Hurewicz homomorphism is an isomorphism.
\end{proof}

\begin{proof}[Proof of Proposition \ref{prop:ITC-arbitrarily-large}]
Choose $\Sigma$ to be an essential $(n-2)$-homology sphere. Such manifolds are known to be of LS category $n-1$, see \cite[Theorem~4.1]{Katz2006}. Furthermore, by \cite[Theorem~3]{Kervaire1969}, there exists a (unique) homotopy $n$\nobreakdash-\hspace{0pt}sphere $\Sigma_h$ such that the connected sum $\Sigma\#\Sigma_h$ bounds a smooth contractible manifold. In order to conclude the proof carry out the construction outlined in Example \ref{ex:highTC}, starting from $\Sigma\#\Sigma_h$.
\end{proof}


Let us conclude with the following problem.

\begin{problem*}
Is it true that if $S^n$ is a smooth $\zmod{p}$\nobreakdash-\hspace{0pt}sphere with non-empty and path-connected fixed point set, then $S^n$ is equivariantly equivalent to a linear $\zmod{p}$\nobreakdash-\hspace{0pt}sphere if and only if $2 \leq TC^{\zmod{p}}(S^n) \leq 3$?
\end{problem*}

\noindent Of course the question is whether the ``($\Leftarrow$)'' implication holds true. By Smith Theory, the fixed point set $(S^n)^{\zmod{p}}$ is a $\textnormal{mod}\,p$ homology sphere, so in order to justify the statement in question it would suffice to prove that topological complexity of any $\textnormal{mod}\,p$ homology sphere is at least $4$. However, there exist simply-connected $\textnormal{mod}\,p$ homology spheres of LS category~$3$. At present we do not know how to attain the required bound (or whether it is possible at all) for such spaces.
\appendix

\section{Nowhere-vanishing equivariant vector fields}

Let $G$ be a finite group. Write $\textsc{ccs}(G)$ for the set of conjugacy classes of subgroups of $G$. Given a subgroup $H \subseteq G$, denote by $W_G H = N_G H/H$ the \textit{Weyl group} of $H$ in $G$. Furthermore, for a $G$\nobreakdash-\hspace{0pt}space $X$, set
\begin{equation*}
\begin{split}
X^{>H} &= \{ x \in X \;|\; H \varsubsetneq G_x \},
\end{split}
\end{equation*}
where $G_x$ stands for the isotropy group of $x \in X$.

The \textit{equivariant Euler characteristic} of a finite $G$\nobreakdash-\hspace{0pt}CW-complex $X$ is defined by
\[ \chi^G(X) = \sum_{(H) \in \textsc{ccs}(G)} \sum_{n=0}^{\dim X} (-1)^n \nu\big(n, (H)\big)[G/H], \]
where $\nu\big(n, (H)\big)$ denotes the number of $n$\nobreakdash-\hspace{0pt}cells of type $(H)$. (This is an element of the Burnside ring of $G$.) It is not difficult to see that $\chi^G$ can be expressed in terms of classical Euler characteristic in the following way:
\begin{equation}\label{equiv_Euler}
\chi^G(X) = \sum_{(H) \in \textsc{ccs}(G)} \chi\left((X^H, X^{>H})\big/W_G H\right)[G/H].\tag{$\star$}
\end{equation}

Recall that a $G$\nobreakdash-\hspace{0pt}space $X$ satisfies the \textit{weak gap hypothesis} if
\[ \dim X^{>G_x} \leq \dim X^{G_x}-2 \textnormal{ for any $x \in X$.} \]

\begin{thm}[{\cite[Remark 6.8]{Luck2003}}]
Let $G$ be a finite group, $M$ a closed $G$\nobreakdash-\hspace{0pt}manifold. If $\chi^G(M)=0$ and $M$ satisfies the weak gap hypothesis, then $M$~admits a nowhere-vanishing equivariant vector field.
\end{thm}

\begin{cor}\label{vfields}
Let $p$ be a prime. Any odd-dimensional smooth $\zmod{p}$\nobreakdash-\hspace{0pt}sphere~$S^n$ with $(S^n)^{\zmod{p}}$ also an odd-dimensional sphere admits a nowhere-vanishing equivariant vector field.
\end{cor}

\begin{proof}
If the action is trivial, the claimed result is known classically. Assume that the action is non-trivial, so that the codimension of the fixed point set is at least $2$. The weak gap hypothesis amounts to two inequalities:
\begin{enumerate}
\item [(1)]$\dim\,(S^n)^{\zmod{p}} \leq \dim S^n - 2$,
\item [(2)]$\dim \emptyset \leq \dim\, (S^n)^{\zmod{p}}-2$,
\end{enumerate}
both of which are trivially true. It remains to verify that $\chi^{\zmod{p}}(S^n)=0$, but this follows immediately from~(\ref{equiv_Euler}).
\end{proof}

\noindent\noindent\textbf{Acknowledgements.} Our approach to Example \ref{ex:highTC} was inspired by \cite{GomezLarranaga1992} where it was shown that $3$-dimensional non-simply-connected homology spheres have LS category $4$. From our point of view that paper is now superseded by \cite{Dranishnikov2008}, but we feel that a reference is in order. We would also like to express gratitude to W. Marzantowicz, for bringing problems tackled in this paper to our attention, and the referee, whose suggestions improved the manuscript.

\bibliographystyle{amsalpha}
\bibliography{Equivariant-TC.bib}

\end{document}